\documentclass[ECP]{ejpecp} 

\usepackage{enumerate}
\usepackage{times}
\usepackage{eurosym}
\usepackage{url}
\usepackage{bm}

\SHORTTITLE{The martingale property in the context of SDEs} 

\TITLE{The martingale property in the context of\\stochastic differential equations} 

\AUTHORS{%
  Johannes~Ruf\footnote{Department of Mathematics, University College London, UK.
    \EMAIL{j.ruf@ucl.ac.uk}}}
 

\AMSSUBJ{60G44 ; 60H10} 

\SUBMITTED{April 14, 2014} 
\ACCEPTED{April 21, 2015} 

\ARXIVID{1306.0218}

\VOLUME{0}
\YEAR{2015}
\PAPERNUM{0}
\DOI{vVOL-PID}

\ABSTRACT{This note studies the martingale property of a nonnegative, continuous local martingale $Z$, given as a nonanticipative functional of a solution to a stochastic differential equation.   The condition states that $Z$ is a (uniformly integrable) martingale if and only if an integral test of a related functional holds. }

\newcommand{\E}{{\mathbb{E}}}

\newcommand{\R}{{\mathbb{R}}}

\renewcommand{\P}{{\mathbf{P}}}
\newcommand{\Q}{{\mathbf{Q}}}

\newcommand{\1}{\mathbf{1}}

\newcommand{\N}{\mathbb{N}}

\newcommand{\tr}{^\mathsf{T}}

\newcommand{\dd}{\mathrm{d}}
\newcommand{\lc}{[\![}

\begin{document}



\section{Introduction}

This note addresses Girsanov's question\footnote{See the remark on page~296 in \cite{Girsanov_1960}.} in the context of (not necessarily one-dimensional) solutions to stochastic differential equations: Under which conditions is a stochastic exponential a true martingale?  The condition provided here is of probabilistic nature and both sufficient and necessary. It relates the martingale property of a local martingale to the almost sure finiteness of a certain integral functional under a related measure.

To illustrate the condition informally, assume for a moment that the stochastic differential equation
\begin{align*}
	\dd X_t = b(t,X) \dd t + \sigma(t, X) \dd W_t, \qquad X_0 = x_0
\end{align*}
has a  weak solution $X$, defined on some probability space, for some progressively measurable functionals $b$ and ${\sigma}$.  Consider a progressively measurable functional $\mu$ and the corresponding nonnegative local martingale $Z,$ given by
\begin{align*}  
Z_t &:= \exp\left( \int_{0}^{t}\mu( s,X )\tr \mathrm{d} X_s -\int_{0}^{t} \left(\frac{1%
}{2} \mu( s,X)\tr a(s,X) \mu(s,X) + \mu(s,X)\tr b(s,X)\right) \mathrm{d}%
s\right),
\end{align*}
where $a = \sigma \sigma\tr$.  (Below we will generalize this setup to allow $X$ to explode and $Z$ to hit zero.)

First, Proposition~\ref{P existence sde} below shows that the stochastic differential equation
\begin{align*}
	\dd Y_t = \left(b(t,Y) + a(t, Y)  \mu(t,Y)\right)\dd t + \sigma(t, Y) \dd W_t, \qquad Y_0 = x_0
\end{align*}
also has a weak solution $Y$, at least up to the first time that the process $$K := \int_{0}^\cdot \mu( s,Y)\tr a(s,Y) \mu(s,Y) \mathrm{d} s$$ explodes.  Indeed, if $Z$ is a uniformly integrable martingale then Girsanov's theorem, applied to the Radon-Nikodym derivative $Z_\infty$, yields directly that a weak solution $Y$ exists and  that $K$ has probability zero to explode. This note shows that the reverse direction also holds; namely, if the process $K$ with an appropriate choice of weak solution $Y$ does not explode, then $Z$ is a uniformly integrable martingale. 
We refer the reader to Theorem~\ref{T 2} below for the precise statement.

The conditions in this note are sharp and hold under minimal assumptions but are purely probabilistic and, in particular, often require additional existence and uniqueness results to be applicable.  Two  examples in Section~\ref{S:ex} illustrate these subtle points. A third example highlights  the relevance of the underlying probability space.

\subsection*{Related literature}
The conditions in \cite{Kabanov/Liptser/Shiryaev:1979},  
\cite{Engelbert_Senf_1991}, \cite{BenAri}, and \cite{Blei_Engelbert_2009} are closely related to those discussed here, as they also involve the explosiveness of the quadratic variation of the local martingale's stochastic logarithm.   \cite{CFY} also studies the martingale property in the context of a martingale problem. \cite{McKean_1969}, \cite{Elworthy2010}, and \cite{Karatzas_Ruf_2013} work out a precise relationship between explosions of solutions to stochastic differential equations and the martingale property of related processes.

 \cite{Engelbert_Schmidt_1984} provides  analytic conditions on the functionals $b, \sigma$, and $\mu$ for the martingale property of the local martingale $Z$, in the context of time-homogeneous conditions.
 \cite{Stummer_1993} provides further analytic conditions if the dispersion function is the identity. In the one-dimensional case, a full analytic characterization of the  martingale property of $Z$ is provided by \cite{MU_martingale}.
In the specific setup of ``removing the drift,'' \cite{Rydberg_1997} and, in the context of stochastic volatility models, \cite{Sin} give easily verifiable conditions. \cite{Blanchet_Ruf_2012} describes a methodology to decide on the martingale property of a nonnegative local martingale, based on weak convergence considerations. For further pointers to a huge amount of literature in this area, we refer the reader to \cite{Ruf_Novikov}.

\section{Setup}
We now formally introduce the setup of this work. We first consider a specific martingale problem whose solution $\P$ is the starting point of our analysis. We then introduce a nonnegative $\P$--local martingale $Z$.  In Section~\ref{S:main} we shall then study a necessary and sufficient condition that $Z$ is a (uniformly integrable)  $\P$--martingale.

\subsection{Generalized local martingale problem}  \label{S:2.1}
Fix $d \in \N$, an open set $E\subset {\mathbb{R}}^{d}$, and a ``cemetery state''  $\Delta \notin \R^d$. Let   $\Omega$  denote the set of all these paths $\omega: [0,\infty) \rightarrow E \bigcup \{\Delta\}$ such that  $\omega(t) = \omega(t \wedge \bm{\zeta}(\omega))$ and  $\omega$ is continuous on $[0, \bm{\zeta}(\omega))$, where
$$\bm{\zeta}(\omega) :=\inf \{t \geq 0\mid \omega(t) = \Delta\}.$$
Here and in the following we use the convention $\inf \emptyset := \infty$.
Let $X$ denote the canonical process and $\mathbb{M}=(\mathcal{M}_t)_{t \geq 0}$ the right-continuous modification of the natural filtration generated by $X$  and set $\mathcal M := \mathcal M_\infty := \bigvee_{t \geq 0} \mathcal M_t$.  For all  closed sets $F \subset E$, introduce the stopping times  
$$\bm \rho_{F}:=\inf \{t \geq 0  
 \mid X_t \notin F\}.$$
For a probability measure $\P$ on  $(\Omega, \mathcal{M})$ and a stopping time $\bm \eta$, the measurable mapping $s: \Omega \rightarrow \Omega,\, \omega \mapsto \omega(\cdot \wedge \bm \eta)$ induces the push-forward measure $\P^{\bm \eta}$, given by $\P^{\bm \eta}(\cdot) = \P(s^{-1}(\cdot))$.  Similarly, for a stochastic process $Y$ and a stopping time ${\bm \eta}$ we write $Y^{\bm \eta}$ to denote the stopped version of $Y$; that is, $Y^{\bm \eta}_t = Y_{{\bm \eta} \wedge t}$ for each $t \geq 0$.

Call a function $g: [0,\infty )\times \Omega \rightarrow \R^n$, for some $n \in \N$, \emph{progressively measurable} if $g$,  restricted to $[0,t] \times \Omega$, is $\mathcal B([0,t]) \otimes \mathcal M_t$--measurable for each $t \geq 0$.
For example, the function $g$ is progressively measurable if $g(\cdot,\mathrm x) = \mathbf{g}(\mathrm x(\cdot))$ for all $\mathrm x \in \Omega$, where $\mathbf{g}: E \bigcup \{\Delta\} \rightarrow \R$ is measurable.

The next definition is in the spirit of Section~1.13 in \cite{Pinsky}:
\begin{definition}[Generalized local martingale problem]
\label{D generalized} 
Fix  an initial point $x_0 \in E$.  
Let  ${a}: [0,\infty)\times \Omega \rightarrow{\mathbb{R}}^{d \times d}$  and
${b}:[0,\infty)\times  \Omega \rightarrow{\mathbb{R}}^{d}$ 
 denote two progressively measurable functions such that the function ${a}$ is symmetric and non-negative definite.
\begin{itemize}
\item We call a probability measure $\P$ on  $(\Omega, \mathcal{M})$ a solution to the
 generalized local martingale problem corresponding to the quadruple $(E, x_0, {a}, {b})$ if $\P(X_0 = x_0) = 1$ and
there exists a nondecreasing sequence $(E_n)_{n \in \N_0}$ of closed subsets of $E$ with $E = \bigcup_{n \in \N_0} E_n$ such that $\P(\bm\rho_{E_n} = \bm\zeta<\infty) = 0$  and 
\begin{align*}
 f\left(X^{\bm\rho_{E_{n}}}_\cdot\right) - \int_{0}^{\cdot\wedge \bm\rho_{E_{n}}} \left(
\sum_{i=1}^{d} {b}_{i}(t,X) f_{x_{i}}(X_t) + \frac{1}{2}
\sum_{i,j=1}^{d} {a}_{i,j}(t,X) f_{x_{i},x_{j}}(X_t) \right) \mathrm{d} t
\end{align*}
is a $\P$--local martingale for each $n \in\mathbb{N}_0$ and twice continuously differentiable function $f: E \rightarrow \R$ 
with partial
derivatives $f_{x_{i}}$ and $f_{x_{i},x_{j}}$.
\item Given a stopping time ${\bm \eta}$
 we say that
a probability measure $\P$ is a solution to the
 generalized local martingale problem corresponding to the quadruple $(E, x_0, {a}, {b})$ on $\lc 0, {\bm \eta}\lc$ if  there exists a nondecreasing sequence of stopping times $({\bm \eta}_n)_{n \in \N}$ with $\lim_{n \uparrow \infty} {\bm \eta}_n = {\bm \eta}$, $\P$--almost surely, such that the push-forward measure $\P^{{\bm \eta}_n}$ is a solution to the generalized martingale problem corresponding to the quadruple $(E, x_0, {a}^n, {b}^n)$, for each $n \in \N$. Here,  ${a}^n(t,\mathrm x) :=  {a}(t,\mathrm x) \1_{t<{\bm \eta}_n(\mathrm{x})}$ and ${b}^n(t,\mathrm x) :=  {b}(t,\mathrm x) \1_{t<{\bm \eta}_n(\mathrm{x})}$ for all $(t, \mathrm x) \in [0,\infty) \times \Omega$.  \qed
\end{itemize}
\end{definition}

Observe that the initial point $x_0$ is fixed in Definition~\ref{D generalized}; in particular, the solution to a generalized local martingale problem here is not a family of probability measures indexed over the initial point, but one probability measure only. See, for example, \cite{Engelbert_2000} for this subtle point.  This weaker requirement allows us to apply the characterization of this note to a larger class of processes.

Throughout this note,  fix $d \in \N$, an open set $E\subset {\mathbb{R}}^{d}$, an initial point $x_0 \in E$,  and progressively measurable functions
$b: [0,\infty)\times\Omega  \rightarrow{\mathbb{R}}^{d}$ and $a: [0,\infty)\times \Omega \rightarrow{\mathbb{R}}^{d \times d}$, such that $a$ is symmetric and nonnegative definite. 
We shall work under the following assumption:

\begin{assumption} There exists a solution $\P$ to the generalized local martingale problem corresponding to $(E,x_0, a,b)$.
\qed
\end{assumption}

Various sufficient conditions for this standing assumption to hold are provided in Section~1.2 of \cite{Cherny_Engelbert} and in Sections~1.7--1.14 of \cite{Pinsky}.

\subsection{A nonnegative local martingale}  \label{S:2.2}
In this subsection, we introduce a $\P$--local martingale $Z$ as a stochastic exponential.  Towards this end, we fix a progressively measurable function $\mu:[0,\infty)\times\Omega  \rightarrow{\mathbb{R}}^{d}$ and make the following assumption:
\begin{assumption}  \label{A2}  We have
\[
\pushQED{\qed} 
	\P\left(\text{the function  $\, \, [0,\infty) \ni t \mapsto \int_0^{t \wedge \bm \zeta} \mu(s, X)\tr a(s,X) \mu(s,X) \dd s\,\,$ jumps to $\infty$}    \right) = 0.    \qedhere 
	\popQED
\]
\end{assumption}

Recall now  the nondecreasing sequence $(E_n)_{n \in \N_0}$ of Definition~\ref{D generalized} and consider the stopping times 
$$\widetilde{\bm{\tau}}_n  :=\inf \left\{t \geq 0 \left|   \int_0^{t} \mu(s, X)\tr a(s,X) \mu(s,X) \dd s > n\right.\right\}, \qquad \bm \theta_n := \widetilde{ \bm \tau}_n \wedge \bm \rho_{E_n} \wedge n$$ 
for all $n \in \N_0$, and $\bm\theta := \lim_{n \uparrow \infty} \bm\theta_n$.  Observe that  $\P(\bm\theta_n < \bm\theta)=1$ for all $n \in \N_0$ thanks to Standing Assumption~\ref{A2}. Therefore, the nondecreasing sequence $( \bm\theta_n)_{n \in \N}$ of stopping times announces $\bm \theta$.

Next, the processes $$M^n := \int_0^{\cdot \wedge \bm \theta_n}  \mu(s,X)   \tr \mathrm d \left(X(s) - \int_0^s b(t,X) \mathrm d t\right)$$ are well defined and indeed uniformly integrable $\P$--martingales, for all $n \in \N_0$. Moreover, for all $m,n \in \N_0$ with $m \leq n$, we have $M^m \equiv (M^n)^{\bm \theta_m}$, and thus, we may ``stick them together'' to obtain the process 
$$M := \sum_{n =1}^\infty M^n \1_{\lc \bm \theta_{n-1}, \bm \theta_n \lc},$$ 
which satisfies $M^{\bm \theta_n} \equiv M^n$ for all $n \in \N_0$ and thus, is a local martingale on $\lc 0, \bm \theta\lc$.
To provide some intuition, the process $M$ is the stochastic integral of the process $\mu(\cdot,X)$ with respect to the local martingale part of $X$ up to the first time that either $X$ or the stochastic integral explodes.
We also introduce the process   $\langle M \rangle$ by
 \begin{align*}  
\langle M\rangle_t &:=   \int_0^{t \wedge \bm \theta} \mu(s, X)\tr a(s,X) \mu(s,X) \dd s
\end{align*}
for all $t \geq 0$.

 Now, define the nonnegative process  $Z$ by
\begin{align}  \label{E T2 M}
Z_t &:= \exp\left(M_t -\frac{1}{2} \langle M\rangle_t\right)   \,\,  \text{for all $t < \bm \theta$} \qquad \text{and} \qquad Z_t := \lim_{s \uparrow \bm \theta } Z_s  \,\, \text{for all $t \geq \bm \theta$}. 
\end{align}
By the supermartingale convergence theorem, the limit always exists and  the process $Z$ is a nonnegative
continuous  $\P$--local martingale; see also  Lemma~4.14 and Appendix~A in  \cite{Ruf_Larsson}.   Consider now the stopping times
\begin{align*}
	{\bm{\tau}_n}   :=\inf \left\{t \geq 0 \left|   \langle M \rangle_t > n\right.\right\}
\end{align*}
for all $n \in \N_0$. Then we have ${\bm{\tau}_n}  \geq \widetilde{\bm{\tau}}_n$ and Novikov's condition yields that the  $\P$--local martingale $Z^{\bm{\tau}_n}$ is a uniformly integrable $\P$--martingale for each $n \in \N$. 

\section{Main result} \label{S:main}

We are interested in finding a necessary and sufficient condition for the nonnegative  $\P$--local martingale $Z$ to be a true  $\P$--martingale. The condition in this note is probabilistic in nature and is formulated under a certain probability measure that is 
a solution 
to the generalized local martingale problem corresponding to $(E,x_0,a, \widehat{b})$ on $\lc 0, \bm{\theta}\lc$, where
\begin{align*}
	 \widehat{b}(t,\mathrm x) := b(t,\mathrm x) + a(t,\mathrm x) \mu(t,\mathrm x) 
\end{align*}
for all $(t,\mathrm x) \in [0,\infty) \times \Omega$.
Note that if $Z$ is a uniformly integrable  $\P$--martingale then a solution to this generalized local martingale problem is given by $\Q$, defined by $\dd \Q =  Z_\infty \dd \P$, thanks to Girsanov's theorem.
The following result yields that a solution to this generalized local martingale problem exists even if $Z$ is not a  $\P$--martingale:

\begin{proposition}[Existence of a solution to the related martingale problem]
\label{P existence sde} 
The generalized local martingale problem corresponding to $(E,x_0,a, \widehat{b})$ on $\lc 0, \bm{\theta}\lc$  has a solution $\Q$ that also satisfies $(\dd \Q|_{\mathcal M_{\bm \tau_n}}) / (\dd \P|_{\mathcal M_{\bm \tau_n}})  = Z^{\bm \tau_n}_\infty$ for each $n \in \N$.
\end{proposition}

\begin{proof}
For any stopping time $\bm \eta$, define the sigma algebra
\begin{align*}
	\mathcal M_{\bm \eta-} := \sigma(X_0) \vee \sigma\left\{A \cap \{\bm \eta > t\}  \mid    A \in \mathcal M_t, t \geq 0\right\}.
\end{align*}
Here, $\sigma(X_0) \subset \mathcal M_0$ denotes the sigma algebra generated by $X_0$. 

Define now the sequence  $(\Q_n)_{n \in \N}$ of probability measures by $\mathrm{d}\Q_{n}=Z^{\bm{\tau}_n}_\infty\mathrm{d} \P$ and observe that $\Q_n(A) = \Q_m(A)$ for all $A \in \mathcal{M}_{(\bm{\tau}_n \wedge \bm{\tau}_m)-}$ and $n,m \in \N$.  Thus, the set function $\Q: \bigcup_{n \in \N} \mathcal{M}_{\bm{\tau}_n-} \rightarrow [0,1]$ with $A \mapsto \Q_n(A)$ for all $A \in \mathcal{M}_{\bm{\tau}_n-}$ is well defined. 
A standard extension theorem, such as Theorem~V.4.1 in \cite{Pa}, then yields that $\Q$ can be extended to a probability measure on $ \bigvee_{n \in \N} \mathcal{M}_{\bm{\tau}_n-}$; see also \cite{F1972} or Appendix~B in \cite{CFR2011}.  
We  may now extend this measure to a probability measure on $(\Omega,\mathcal{M})$; see Theorem~E.2 in  \cite{Perkowski_Ruf_2014} and use $\bigvee_{n \in \N} \mathcal{M}_{\bm{\tau}_n-} = \mathcal{M}_{(\lim_{n \uparrow \infty} \bm{\tau}_n)-}.$
With a slight misuse of notation, we again write $\Q$ for this probability measure, constructed via an extension argument.  

Next, fix $n \in \N$ and $A \in \mathcal M_{\bm \tau_n}$ and note that $\bm \tau_n(\omega) <  \bm \tau_{n+1}(\omega) $ on $\{\bm \tau_n<\infty\}$ since $\langle  M \rangle(\omega)$ is continuous and does not jump to infinity, for any $\omega \in \Omega$, by construction of the stopping time $\bm \theta$.  This then yields 
\begin{align*}
	\Q(A) &=  \Q(A \cap \{\bm\tau_n < \infty\}) + \Q(A \cap \{\bm \tau_n = \infty\}) \\&=  \E^{\P}\left[Z^{\bm \tau_{n+1}}_\infty \1_{A \cap \{\bm \tau_n < \infty\}}\right] + \E^{\P}\left[Z^{\bm \tau_{n}}_\infty \1_{A \cap \{\bm \tau_n = \infty\}}\right]
		=  \E^{\P}\left[Z^{\bm \tau_{n}}_\infty \1_A\right]
\end{align*}
since $A \cap \{\bm \tau_n = \infty\} \in  \mathcal M_{\bm \tau_n-}$.
The statement then follows.
\end{proof}

Note that it is a common approach to use a change of measure to prove the existence of a solution to a given martingale problem, as in the proof of Proposition~\ref{P existence sde}; see, for example, \cite{SV_multi}.  However, usually only \emph{equivalent} changes of measures are considered.

\begin{remark}  \label{R uniqueness}
	Observe that Proposition~\ref{P existence sde} does not make any assertion concerning the uniqueness of the measure $\Q$.  In general, such uniqueness does not hold. However, after fixing a probability measure $\P$ from the set of solutions to the  generalized local martingale problem corresponding to  $(E,x_0,a, b)$, the probability measure $\Q$ of Proposition~\ref{P existence sde} is uniquely determined on $\bigvee_{n \in \N} \mathcal{M}_{\bm{\tau}_n}$. 
\qed
\end{remark}

We are now ready to state a characterization of the martingale property of the  $\P$--local martingale $Z$:
\begin{theorem}[Characterization of martingale property]
\label{T 2} 
With  $\Q$ denoting the measure of Proposition~\ref{P existence
sde}, the following equivalences hold: The  $\P$--local martingale $Z$, given in \eqref{E T2 M}, is a $\P$--martingale if and only if 
\begin{align} \label{E suffCond} 
	\Q\left(\int_{0}^{t \wedge \bm \theta} \mu( s,X)\tr a(s,X) \mu(s,X)\mathrm{d} s < \infty\right) = 1
\end{align}
for all $t \geq 0$.  The  $\P$--local martingale $Z$ is a uniformly integrable $\P$--martingale if and only if 
\begin{align} \label{E suffCond2} 
	\Q\left(\int_{0}^{ \bm\theta} \mu( s,X)\tr a(s,X) \mu(s,X)\mathrm{d} s < \infty\right) = 1.
\end{align}
\end{theorem}

\begin{proof}
	We start by assuming that $Z$ is a  $\P$--martingale. We need to show that  $\Q(A_n) = 0$ for the nondecreasing sequence of  events $(A_n)_{n \in \N}$, defined by 
	\begin{align*}
		A_n :=  \left\{ \langle M\rangle_{n \wedge \bm{\theta}} = \infty\right\} = \bigcap_{k \in \N} \{\bm \tau_k < n \wedge \bm \theta \} \in \mathcal{M}_{(n \wedge \bm{\theta})-}
	\end{align*}
for all $n \in \N$. Fix $n \in \N$ and observe  that the martingale property of $Z$ yields a measure $\Q^Z$, defined by $\dd \Q^Z = Z_n \dd \P= Z_{n \wedge \bm{\theta}} \dd \P$.  Since $ \mathcal{M}_{(n \wedge \bm{\theta})-} = \bigvee_{m \in \N} \mathcal{M}_{(n \wedge \bm{\tau}_m \wedge \bm\theta)-}$, it is easy to see that $\Q^Z|_{\mathcal{M}_{(n \wedge \bm{\theta})-}} = \Q|_{\mathcal{M}_{(n \wedge \bm{\theta})-}}$. Thus, we have $\Q(A_n) = \Q^Z(A_n) = \E^\P[Z_n \1_{A_n}] = 0$ since $Z_n = 0$ $\P$--almost surely on $A_n$ by the Dambis-Dubins-Schwarz theorem.

For the reverse direction,   note that
\begin{align*}
	\Q(A) = \lim_{n \uparrow \infty} \Q(A \cap \{\bm \tau_n > t\wedge \bm \theta\}) \leq  \lim_{n \uparrow \infty} \E^\P\left[Z^{ \bm \tau_n}_\infty  \1_A  \right] = 0
\end{align*}
for all $t \geq 0$ and $A \in \mathcal M_{t \wedge \bm \theta}$ with $\P(A) = 0$.  Here, we have used the assumption, namely that \eqref{E suffCond}  holds, in the first equality.
Thus, $\Q$ is absolutely continuous with respect to $\P$ on $\mathcal M_{t \wedge \bm \theta}$ for each $t \geq 0$.
Define now the  $\P$--martingale $R$ by $$R_t := \frac{\dd \Q|_{\mathcal M_{t \wedge \bm \theta}}}{\dd \P |_{\mathcal M_{t \wedge \bm \theta}}}$$ for each $t \geq 0$.  The fact that $R^{\bm\tau_n \wedge n} \equiv Z^{\bm \tau_n \wedge n}$ for each $n \in \N$ and taking limits then yield $R \equiv Z$. Thus, $Z$ is a $\P$--martingale.

The second equivalence is proven in the same way.
\end{proof}

We refer the reader to \cite{Musiela_1986}, \cite{Engelbert_Senf_1991}, \cite{Khos_Salminen_Yor}, and \cite{MU_integral} for analytic conditions that yield \eqref{E suffCond} in the case $d=1$. 
Theorem~\ref{T 2} extends Theorem~1 in \cite{BenAri} to a bigger class of stochastic differential equations; moreover, Proposition~\ref{P existence sde} yields that one does not need to assume the existence of the measure $\Q$, as it always exists.  We remark that in the one-dimensional time-homogeneous case, under some additional regularity conditions, an analytic characterization of the martingale property of $Z$ has been obtained; most notably, by   \cite{MU_martingale}. This characterization is given in terms of the
behavior of $X$ under $\P$ and $\Q$ at the boundary points of the
one-dimensional interval $E$.



\begin{corollary}[Pathwise integrability]  \label{C:1}
If 
\begin{align*} 
\int_{0}^{{t \wedge \bm \theta}(\mathrm x)} \mu(s, \mathrm x)\tr a(s,\mathrm x) \mu(s,\mathrm x)  \mathrm{d} s < \infty 
\end{align*}
holds  for all $(t,\mathrm x) \in[0,\infty) \times\Omega$ then $Z$ is a  $\P$--martingale.
Moreover, if 
\begin{align*} 
\int_{0}^{{\bm \theta}(\mathrm x)} \mu(s, \mathrm x)\tr a(s,\mathrm x) \mu(s,\mathrm x)  \mathrm{d} s < \infty 
\end{align*}
holds  for all $\mathrm x \in \Omega$ then $Z$ is a uniformly integrable $\P$--martingale.
\end{corollary}
\begin{proof}
	The statement follows directly from Thereom~\ref{T 2}.
\end{proof}

\begin{remark}
We emphasize certain caveats concerning Theorem~\ref{T 2}:
\begin{itemize}
\item The choice of a solution to the generalized local martingale problem corresponding to $(E,x_0,a, b)$   matters for the question whether the  local martingale $Z$ is a  martingale.  Indeed, as Example~\ref{ex1} illustrates, the  local martingale $Z$ might be a true  martingale under one measure and a strict local martingale under another measure.  
\item However, the choice of  measure $\Q$ among the ones that satisfy the conditions in Proposition~\ref{P existence sde}, namely the ones that agree on $\bigvee_{n \in \N} \mathcal M_{\bm \tau_n}$, is not relevant. This is due to the fact that  \eqref{E suffCond} and \eqref{E suffCond2} hold either for all such probability measures with the prescribed ``local'' distribution or for none. (See also Remark~\ref{R uniqueness}.)
\item The generalized local martingale problem corresponding to $(E,x_0,a, \widehat{b})$ might have a solution that is unique  among the subset of non-explosive solutions, but that is not unique among all solutions.  Nevertheless, Theorem~\ref{T 2} may be applied, but  the probability measure $\Q$ needs to be chosen carefully. See Example~\ref{ex2} for an illustration.
\item We have not assumed that the  $\P$--local martingale $Z$ is strictly positive. For example, consider the parameter constellation $d=1$, $E = (0,\infty)$, $x_0=1$, and
$$a(t,\mathrm x) = 1, \qquad b(t,\mathrm x) =  0, \qquad \mu(t,\mathrm x) = \mathbf{1}_{\bm \zeta(\mathrm x) > t } \frac{1}{\mathrm x(t)}$$ for all $(t, \mathrm x) \in [0,\infty) \times \Omega$.  The solution to the  generalized local martingale problem corresponding to $(E, x_0, a, b) = ((0,\infty),1,1,0)$ then is Brownian motion killed when hitting zero and is unique. In particular, the stopping time $\bm \theta$ of Subsection~\ref{S:2.2} is the first time that the Brownian motion leaves $E$; that is, $\bm \theta = \bm \zeta$. Note that the  $\P$--local martingale $Z$ is a  $\P$--Brownian motion stopped in zero.  Now, under $\Q$,  the unique solution to the generalized local martingale problem corresponding to $(E, x_0, a, \mu)$, the process $X$ is a three-dimensional Bessel process. In particular, argued for example via Feller's test of explosions, we have $\Q(\bm \zeta = \infty) = 1$ and thus
\begin{align*}
	\Q\left(\int_0^t \frac{1}{X_s^2} \dd s < \infty\right) = 1
\end{align*} 
for all $t \geq 0$, which yields \eqref{E suffCond}.  However,   \eqref{E suffCond2} fails. Thus we obtain the obvious statement that the  $\P$--local martingale $Z$ is a true  $\P$--martingale, but not uniformly integrable.
\item The statement of Corollary~\ref{C:1} is wrong, in general, if we replace the underlying filtered space $(\Omega, \mathcal M, \mathbb M)$ by the space of $E$--valued continuous paths, along with the right-continuous modification of the canonical filtration.  This is illustrated in Example~\ref{ex3}.
\qed
\end{itemize}
\end{remark}

\section{Examples}  \label{S:ex}
The examples of this section illustrate the subtle points in the application of Theorem~\ref{T 2} and Corollary~\ref{C:1}.

\begin{example}[Non-uniqueness]  \label{ex1}
	Let $d=1$, $E = (0,\infty),$ and $x_0 = 1$.  
	Set 
\begin{align*}
	a(t,\mathrm x) =  \1_{\mathrm x(t) \neq 1}, \qquad b(t, \mathrm x) = \mathbf{1}_{\bm \zeta(\mathrm x) > t }   \1_{\mathrm x(t) \neq 1}  \frac{1}{\mathrm x(t)}, \qquad \mu(t, \mathrm x) =  -b(t,\mathrm x)
\end{align*}
 for all $(t, \mathrm x) \in [0,\infty) \times \Omega$.
 The generalized local martingale problem corresponding to the quadruple $(E, x_0, a, b)$ has a solution $\P_1$; indeed $\P_1(X_\cdot \equiv 1)=1$ satisfies all conditions.  However, the solution is not unique. Another solution $\P_2$ would be the one corresponding to the three-dimensional Bessel process, started in one.  

	Observe that the process $Z$ is  a  local martingale in each case.  In the first case, it is almost surely constant, that is, $\P_1(Z_\cdot\equiv 1) = 1$,  and thus the process $Z$ is a (uniformly integrable) $\P_1$--martingale. In the second case, It\^o's formula yields that $Z$ is distributed as the reciprocal of a three-dimensional Bessel process and thus, a strict $\P_2$--local martingale.

	Consider now the generalized local martingale problem corresponding to  the quadruple $(E, x_0, a, b+a \mu) = ((0,\infty), 1, a, 0)$, which also has a solution according to Proposition~\ref{P existence sde}.  Indeed, it has several solutions, in particular $\Q_1 \equiv \P_1$ and the Brownian motion measure $\Q_2$.  Note that \eqref{E suffCond} with $\Q=\Q_1$ holds but with $\Q=\Q_2$ fails.  This observation is consistent with the fact that $Z$ is a $\P_1$--martingale but a strict $\P_2$--local martingale.
\qed
\end{example}

The next example illustrates that the choice of the probability measure $\Q$ in Theorem~\ref{T 2} is highly relevant if several solutions exist to the  generalized local martingale problem corresponding to  the quadruple $(E, x_0, a, b+a \mu)$.

\begin{example}[Uniqueness of non-explosive solution]  \label{ex2}
	Let $d=1$, $E = \R$, and $x_0 = 0$.  
	Set 
$$a(t,\mathrm x) =1-\1_{\min_{s \leq t} \{\mathrm x(s)\} = 0 = \max_{s \leq t} \{\mathrm x(s)\}}, \qquad b(t, \mathrm x) =  0, \qquad \mu(t, \mathrm x) =  (\mathrm x(t))^2 \mathbf{1}_{\bm \zeta(\mathrm x) > t  }$$ for all $(t, \mathrm x) \in [0,\infty) \times  \Omega$. Again, the
  generalized local martingale problem corresponding to the quadruple $(E, x_0, a, b)$ has several solutions; for example $\P_1$ such that $\P_1(X_\cdot \equiv 0) = 1$ and the Brownian motion measure $\P_2$.

Consider now the generalized local martingale problem corresponding to  the quadruple $(E, x_0, a, b+a \mu) = (\R, 0, a, a \mu)$.  Clearly, it has several solutions, in particular the constant process with $\Q_1 \equiv \P_1$ and, moreover, $\Q_2$, under which $X$ satisfies the stochastic differential equation
\begin{align*}
	X_t =  \int_0^t X_s^2 \dd s + W_t
\end{align*}
for each $t \geq 0$ for some $\Q_2$--Brownian motion $W$ up to an explosion time, which is finite $\Q_2$--almost surely by Feller's test of explosions.  Indeed, it is easy to see that the choice of parameters in this example implies that any solution to the  generalized local martingale problem corresponding to the quadruple $(\R, 0, a, a \mu)$  is a process that is either constant zero or explodes almost surely.  Thus, this generalized local martingale problem has a unique non-explosive solution.

However, note that Theorem~\ref{T 2} does rely on a certain choice of solution $\Q$, which does not always correspond to $\Q_1$.  In particular, $Z$ here is a (uniformly integrable)  $\P_1$--martingale but a strict $\P_2$--local martingale.
\qed
\end{example}

\begin{example}[Role of the underling probability space]\label{ex3}
 We consider now, instead of the filtered space $(\Omega, \mathcal M, \mathbb M)$ of Subsection~\ref{S:2.1} the filtered probability space $(\Omega', \mathcal M', \mathbb M')$, where $\Omega' = C([0,\infty), E)$ denotes the space of $E$--valued continuous paths with canonical process $X'$,  $\mathbb{M}'=(\mathcal{M}_t')_{t \geq 0}$ denotes the right-continuous modification of the natural filtration generated by $X'$, and $\mathcal M' = \bigvee_{t \geq 0} \mathcal M_t'$ denotes the smallest sigma algebra that makes $X'$ measurable. Note that $\Omega' \subsetneq \Omega.$
  Exactly as in Subsection~\ref{S:2.1}, we can now introduce the notions of progressive measurability and solutions $\P'$ to the generalized local martingale problem.  Moreover, given such a solution $\P'$, we can introduce a $\P'$--local martingale $Z'$ exactly as in Subsection~\ref{S:2.2}.
  
    Let now $d=1$, $E=\R$, and $x_0 =0$.  Moreover, for some fixed $T>0$ set $$a(t,\mathrm x') = 1, \qquad b(t,\mathrm x') =  0, \qquad \mu(t,\mathrm x') = (\mathrm x'(t))^2 \1_{t \leq T}$$ for all $(t, \mathrm x') \in [0,\infty) \times \Omega'$.
    Then, there exists a unique solution $\P'$ to the generalized local martingale problem corresponding to $(E,x_0,a,b) = (\R, 0,1,0)$ on $(\Omega',\mathcal M')$. Indeed, $\P'$ corresponds to the Wiener measure on $(\Omega, \mathcal M')$.
    
   Next, the $\P'$--local martingale $Z'$,  given by
   \begin{align*}
   	Z_t' = \exp\left(\int_0^t (X_s')^2 \dd X_s' - \frac{1}{2} \int_0^t (X_s')^4 \dd s \right)
   \end{align*}
for all $t \geq 0$,   is not a $\P'$--martingale (see Section~3.7 in \cite{McKean_1969}). Thus, there exists $u>0$ such that $\E^{\P'}[Z_u]<1$ and we set $T=u$.  Note that
\begin{align*}
	\int_0^\infty \mu(s, \mathrm x')\tr a(s,\mathrm x') \mu(s,\mathrm x')  \mathrm{d} s  = \int_0^T (\mathrm x'(s))^4 \mathrm{d} s   < \infty 
\end{align*}
by continuity of the path $\mathrm x'$, for all $\mathrm x' \in \Omega'$.  This shows that the assertion of  Corollary~\ref{C:1} is wrong, in general, if we replace $(\Omega, \mathcal M, \mathbb M)$  by $(\Omega', \mathcal M', \mathbb M')$.

To understand, why the assumption of Corollary~\ref{C:1} is  not satisfied if we replace $\Omega'$ by $\Omega$ fix the path $\mathrm x \in \Omega \setminus \Omega'$ with $\mathrm x(t) = \tan(t \pi/(2 T)) \1_{t < T} + \Delta \1_{t \geq T}$ for all $t \geq 0$. Then, we have $\int_0^T (\mu(t, \mathrm x))^2 \dd t = \infty$.
\qed
\end{example}


\ACKNO{I am deeply indebted to Jose Blanchet, Zhenyu Cui, Hans-J\"urgen Engelbert, Ioannis Karatzas, Martin Larsson, and Nicolas Perkowski for very helpful comments and discussions on the subject matter of this note.  I am very grateful to an anonymous referee for her or his very careful reading and suggestions, which have improved this note significantly.  I acknowledge generous support from the Oxford-Man Institute of Quantitative Finance, University of Oxford.}

\end{document}